\documentclass{amsart}[12pt]
\usepackage{amsmath,amssymb,amscd,amsthm,indentfirst}
\usepackage{amsfonts}
\usepackage[mathscr]{eucal}

\newtheorem*{claimx}{Claim}

\newtheorem{prop}{Proposition}[section]
\newtheorem{thm}[prop]{Theorem}

\newtheorem{cor}[prop]{Corollary}

\newtheorem{lem}[prop]{Lemma}

\theoremstyle{definition}

\newtheorem{remark}[prop]{Remark}

\def\map{\mathrm{map}}

\def\Aut{\mathrm{Aut}}
\def\Obj{\mathrm{Obj}}

\def\vp{\varphi}
\def\al{\alpha}
\def\be{\beta}
\def\ga{\gamma}

\def\de{\delta}
\def\Ga{\Gamma}

\def\Si{\Sigma}
\def\si{\sigma}

\def\B{\ensuremath{\mathcal{B}}}
\def\C{\ensuremath{\mathcal{C}}}
\def\D{\ensuremath{\mathcal{D}}}
\def\E{\ensuremath{\mathcal{E}}}
\def\F{\ensuremath{\mathcal{F}}}
\def\H{\ensuremath{\mathcal{H}}}

\def\K{\ensuremath{\mathcal{K}}}
\def\G{\ensuremath{\mathcal{G}}}
\def\L{\ensuremath{\mathcal{L}}}

\def\O{\ensuremath{\mathcal{O}}}
\def\P{\ensuremath{\mathcal{P}}}
\def\R{\ensuremath{\mathcal{R}}}

\def\T{\ensuremath{\mathcal{T}}}

\def\Z{\ensuremath{\mathcal{Z}}}

\def\BB{\ensuremath{\mathbb{B}}}
\def\FF{\ensuremath{\mathbb{F}}}

\def\ZZ{\ensuremath{\mathbb{Z}}}

\def\Aut{\mathrm{Aut}}

\def\diag{\operatorname{diag}}

\def\GL{\operatorname{GL}}

\def\SL{\operatorname{SL}}
\def\Hom{\mathrm{Hom}}
\def\incl{\text{incl}}
\def\id{\textrm{id}}

\def\Inn{\operatorname{Inn}}

\def\op{\ensuremath{\mathrm{op}}}

\def\hocolim{\operatorname{hocolim}}

\def\res{\operatorname{res}}

\def\ker{\operatorname{Ker}}

\def\Ab{\mathbf{Ab}}

\def\Gps{\mathbf{Groups}}
\def\Top{\mathbf{Top}}
\def\hotop{\mathbf{hoTop}}
\def\spectra{\mathbf{Spectra}}

\def\tX{\tilde{X}}

\def\Sol{\operatorname{Sol}}
\def\Spin{\operatorname{Spin}}

\def\vertx{\operatorname{vert}}
\def\edge{\operatorname{edge}}

\def\trp[#1,#2,#3]{[\hskip-1.5pt[#1,#2,#3]\hskip-1.5pt]}

\def\darrow{\downarrow}

\newcommand{\xto}[1]{\xrightarrow{#1}}

\newcommand{\pcomp}[1]{{#1}^\wedge_p}

\newcommand{\hhocolim}[1]{\underset{#1}{\hocolim}}

\input xy
\xyoption{all}

\date{\today}
\title{Homology decompositions and groups inducing fusion systems}

\author{Assaf Libman}

\address{
	Department of Mathematical Sciences, 
	King's College,
	University of Aberdeen,
   	Aberdeen AB24 3UE , Scotland, U.K.,
}
\email{a.libman@abdn.ac.uk}

\author{Nora Seeliger}
\address{
L.A.G.A-Laboratoire de Mathematiques
Institut Galilee
99 Avenue J.B. Clement
93430 Villetaneuse
France}

\email{seeliger@math.univ-paris13.fr}

\begin{document}

\pagestyle{plain}


\renewcommand{\thethmain}{\Alph{thmain}}

\renewcommand{\theenumi}{(\arabic{enumi})}
\renewcommand{\labelenumi}{(\arabic{enumi})}

\begin{abstract}
We relate the construction of groups which realize saturated fusion systems and signaliser functors with homology decompositions of $p$-local finite groups.
We prove that the cohomology ring of Robinson's construction is in some precise sense very close to the cohomology ring of the fusion system it realizes.
\end{abstract}

\maketitle


\section{The main results}

The starting point of this paper is the results of Leary-Stancu in \cite{Leary-Stancu07} and Robinson in \cite{Robinson07}.
For any saturated fusion system $\F$ on a finite $p$-group $S$ they give recipes to construct groups $\pi_{LS}$ (for Leary-Stancu) and $\pi_R$ (for Robinson) which are generally infinite and which contain $S$ as a Sylow $p$-subgroup, see \S\ref{subsec:sfs-tansporter-and-linking-systems}, and realize $\F$ in the sense that $\F=\F_S(\pi)$.
Also, for primes $q=3,5 \mod (8)$, Aschbacher and Chermak constructed in \cite{Asch-Chermak} an infinite group $\pi_{AC}$ which contains a Sylow subgroup of $\Spin_7(q^n)$ as its Sylow $p$-subgroup and which realizes Solomon's fusion system $\F_{\Sol}(q^n)$, see \cite{LO1}.
%
%
%

The primary goal of this paper is to give a uniform and conceptual treatment for all these constructions.
Our approach is geometric rather than algebraic.
The main observation in this paper is Theorem \ref{thmA} below which we prove in the end of section \ref{sec:plfgs}.
Also, it was observed by Aschbacher that if a saturated fusion system $\F$ has the form $\F_S(\pi)$ where $\pi$ is possibly infinite, then a signaliser functor on $\pi$, see \S\ref{subsec:sfs-tansporter-and-linking-systems} for details, gives rise to an associated linking system $\L$. 
A partial converse to this observation is point \ref{thmA:signaliser} of Theorem \ref{thmA}.

\begin{thm}\label{thmA}
Fix a $p$-local finite group $(S,\F,\L)$ and let $\pi$ be a group which contains $S$ as a Sylow $p$-subgroup.
Assume that $\F \subseteq \F_S(\pi)$ and that there exists a map $f \colon B\pi \to \pcomp{|\L|}$ whose restriction to $BS \subseteq B\pi$ is homotopic to the natural map $\theta \colon BS \to \pcomp{|\L|}$.
Then
\begin{enumerate}
\item
$\F=\F_S(\pi)$
\label{thmA:induce-fusion}

\item
There is a signaliser functor $\Theta$ on $\pi$ which induces $\L$, namely $\L$ is a quotient of the transporter system $\T^c_S(\pi)$.
\label{thmA:signaliser}

\item
$H^*(|\L|;\FF_p)$ is a retract  of $H^*(B\pi;\FF_p)$ in the category of $\FF_p$-algebras.
It is equal to the image of $H^*(\pi;\FF_p) \xto{\res} H^*(S;\FF_p)$.
\label{thmA:retract}
\end{enumerate}
\end{thm}

The point of Propositions \ref{prop:thmA-conditions-R}, \ref{prop:ls-thmA-conditions} and \ref{prop:ac-thmA} below is that the groups $\pi_{LS}, \pi_R$ and $\pi_{AC}$ satisfy the hypotheses of this theorem, hence recovering the results of \cite{Leary-Stancu07, Robinson07} and \cite{Asch-Chermak}; See Remarks \ref{rem:we-do-as-good}--\ref{rem:cohomology-of-pi} below. 
The conceptual reason that Theorem \ref{thmA} applies to these groups, which is the main message of this paper, is that all these groups are built from different homology decompositions which express $|\L|$ as the homotopy colimit of a diagram of classifying spaces of finite groups.
By taking homotopy colimits of carefully chosen subdiagrams we obtain classifying spaces $B\pi$ together with maps $B\pi \to \pcomp{|\L|}$ where $\pi$ induces the fusion system $\F$. 
The groups $\pi_{LS}$ are built in \S\ref{subsec:Leary-Stancu} from the subgroup decomposition \cite[\S2]{BLO2}.
The groups $\pi_R$ are built in \S\ref{subsec:robinsons-construction} from the normalizer decomposition with respect to the collection of the $\F$-centric subgroups \cite{Lib06}.
The groups $\pi_{AC}$ are related to the normalizer decomposition with respect to the collection of the elementary abelian subgroup of $\F_{\Sol}(q^n)$ as we explain in \S\ref{subsec:Aschbacher-Chermak}.

An interesting conclusion of Theorem \ref{thmA}\ref{thmA:signaliser} is the following result whose proof is given in \S\S\ref{subsec:robinsons-construction}, \ref{subsec:Leary-Stancu}.
It settles a question which was open for some time.

\begin{cor}\label{corC}
Fix a saturated fusion system $\F$ on $S$ and let $\pi$ be either Robinson's group $\pi_R$ or Leary-Stancu's group $\pi_{LS}$ so that $\F=\F_S(\pi)$.
Then the following are equivalent.
\begin{enumerate}
\item
$\F$ has an associated linking system $\L$.
\label{corC:1}

\item
There exists a signaliser functor on $\pi$ which induces $\L$.
\label{corC:2}
\end{enumerate}
\end{cor}


\begin{remark}\label{rem:we-do-as-good}
We explain in Remark \ref{rem:R-avoid-linking-system} below, see also Proposition \ref{prop:stable-fusion}, that the structure of a $p$-local finite group is not essential to deduce point \ref{thmA:induce-fusion} of Theorem \ref{thmA} for arbitrary saturated fusion systems.
Thus, even though Theorem \ref{thmA} is not quite stated in this way, our results are as strong as those of \cite{Leary-Stancu07} and \cite{Robinson07} which apply to arbitrary saturated fusion systems.
\end{remark}

\begin{remark}
Our approach shows that Robinson's original construction isn't quite the natural one.
The group $\pi_R$ that we construct in \S\ref{subsec:robinsons-construction} is a factor group of Robinson's original construction \cite{Robinson07}.
See Remark \ref{rem:robinson-original-vs-pir} below. 
Our costruction of $\pi_R$ is inspired by the results of the first author and Antonio Viruel in \cite{Lib-Vir}.
\end{remark}

\begin{remark}
The construction of $\pi_{AC}$ in \cite{Asch-Chermak} requires $q \equiv 3 \text{ or } 5 \mod (8)$.
This is particularly important for the construction of the signaliser functors.
The group $\pi_{AC}$ that we construct in \S\ref{subsec:Aschbacher-Chermak} is a variation of Aschbacher-Chermak's original construction.
Our group which is motivated by the geometric approach, requires no restriction on $q \mod (8)$, yet it admits a signaliser functor, see Remark \ref{rem:improved-ac}.
Thus, our construction improves Aschbacher-Chermak's albeit at the price of not being as explicit as their's.
\end{remark}

\begin{remark}\label{rem:cohomology-of-pi}
Finally, point \ref{thmA:retract} of Theorem \ref{thmA} applies to $\pi_{LS}, \pi_{R}$ and $\pi_{AC}$ and as far as we know, this cohomological information is new.
\end{remark}

From the cohomological point of view, the groups $\pi_R$ and $\pi_{AC}$ are very close to their ``target'' namely the ring $\H^*(\F)$ of the stable elements, see \cite[Theorem B]{BLO2}, in a sense made precise in the next result which is proven in \S\S\ref{subsec:robinsons-construction}, \ref{subsec:Aschbacher-Chermak}.

\begin{prop}\label{propB}
Let $\pi$ be either the group $\pi_R$ associated to a saturated fusion system $\F$ (see \S\ref{subsec:robinsons-construction}) or the group $\pi_{AC}$ if $\F=\F_{\Sol}(q^n)$ and $p=2$ (see \S\ref{subsec:Aschbacher-Chermak}).
Then 
\begin{enumerate}
\item
$B\pi$ is $p$-good in the sense of \cite[I.\S 5]{Bous-Kan}.

\item
In the kernel of $H^*(B\pi;\FF_p) \xto{Bi^*} H^*(BS;\FF_p)$ the product of any two elements is zero. 
\end{enumerate}
\end{prop}


%
%
%
%
\section{Saturated fusion systems and $p$-local finite groups}\label{sec:plfgs}

In this section we will \emph{very briefly} recall the notion of $p$-local finite groups and of transporter systems.
We will take the approach of Broto Levi and Oliver in \cite{BLO2} and Oliver-Ventura \cite[\S\S 2,3]{OV07}.
Also see \cite{BLO-survey}.
Fusion systems were first defined by Puig, e.g. \cite{Puig06}.

\subsection{Fusion systems}\label{subsec:sfs}
A \emph{fusion system} on a finite $p$-group $S$ is a small category $\F$ whose object set is the set of subgroups of $S$.
Morphism sets $\F(P,Q)$ are group monomorphisms which include the set $\Hom_S(P,Q)$ of conjugations by elements of $S$.
In addition, every morphism $P\to Q$ in $\F$ is the composition of an isomorphism $P \to P'$ in $\F$, which is necessarily a group isomorphism, followed by an inclusion $P' \leq Q$.
From the definitions, every endomorphism in $\F$ is an automorphism.

Subgroups $P, P' \leq S$ are called \emph{$\F$-conjugate} if they are isomorphic as objects of $\F$.
A subgroup $P \leq S$ is called \emph{fully centralized} if $P$ has the largest $S$-centralizer among all the subgroups in its $\F$-conjugacy class.
Similarly, $P$ is \emph{fully normalized} if its $S$-normalizer has the largest possible order among all the groups in its $\F$-conjugacy class.
We say that $\F$ is \emph{saturated} if the following two conditions hold.
First, every fully normalized subgroup $P \leq S$ is fully centralized and in this case $\Aut_S(P)=N_S(P)/C_S(P)$ is a Sylow $p$-subgroup of $\Aut_\F(P)$.
Second, if $\vp \in \F(P,S)$ is such that $\vp(P)$ is fully centralized then $\vp$ extends to a morphism $\tilde{\vp} \in \F(N_\vp,S)$ where $N_\vp \leq N_S(P)$ consists of the elements $g$ such that $\vp \circ c_g \circ \vp^{-1} \in \Aut_S(\vp(P))$ where $c_g$ is the inner automorphism $x \mapsto gxg^{-1}$ of $S$.

A subgroup $P \leq S$ is called \emph{$\F$-centric} if $P$ and all its $\F$-conjugates contain their $S$-centralizer.
We let $\F^c$ denote the full subcategory of $\F$ on the $\F$-centric objects.
A subgroup $P \leq S$ is called $\F$-radical if $\Inn(P)$ is the maximal normal $p$-subgroup of $\Aut_\F(P)$.
Alperin's fusion theorem \cite[Theorem A.10]{BLO2} says that any saturated fusion system $\F$ is generated by the $\F$-automorphism groups of its $\F$-centric $\F$-radical subgroups.
Namely, every morphism in $\F$ is obtained by composition of restrictions of $\F$-automorphisms of these groups.

An important source of saturated fusion systems are finite groups.
If $S$ is a Sylow $p$-subgroup of a finite group $G$, one forms the category $\F_S(G)$ whose objects are the subgroups of $S$ and whose morphisms are the group monomorphisms $c_g \colon P \to Q$ of the form $x \mapsto gxg^{-1}$ for some $g \in G$.


\subsection{Transporter systems, linking systems and signaliser functors}\label{subsec:sfs-tansporter-and-linking-systems}
A finite $p$-subgroup $S$ of a group $\pi$, possibly infinite, is a \emph{Sylow $p$-subgroup} if any finite $p$-subgroup $Q \leq \pi$ is conjugate to a subgroup of $S$.

Define $\F_S(\pi)$ where $S \leq \pi$ is a Sylow $p$-subgroup as we did above for finite groups; Up to isomorphism the choice of $S$ is immaterial.

The objects of the \emph{transporter system} $\T_S(\pi)$ are the subgroups of $S$; its morphism sets are
\[
\T_S(\pi)(P,Q)=N_\pi(P,Q) = \{ g \in \pi \colon gPg^{-1} \leq Q\}.
\]
There is an obvious surjective functor $\T_S(\pi) \to \F_S(\pi)$ which is the identity on objects and is the quotient by the action of $C_\pi(P)$ on morphism sets.
We denote by $\T^c_S(\pi)$ the full subcategory whose objects are the $\F$-centric subgroups of $S$ for some fusion system $\F$ on $S$ which will be understood from the context.

An associated \emph{centric linking system} $\L$ to a saturated fusion system $\F$ on $S$ is a category whose objects are the $\F$-centric subgroups.
It is equipped with a surjective functor $\pi \colon \L \to \F^c$ and an injective functor $\de \colon \T^c_S(S) \to \L$, both induce the identity on object sets.
In addition the following three conditions must be satisfied.
First, the image of $Z(P)$ under $\de \colon N_S(P) \to \Aut_\L(P)$ acts freely on $\L(P,Q)$, and $\F(P,Q)$ is isomorphic to the orbit set via $\pi$.
Second, for any $g \in N_S(P,Q)$ we have $\pi(\de(g))=c_g$.
Third, for any $f \in \L(P,Q)$ and any $g \in P \leq N_S(P)=\Aut_{\T^c_S(S)}(P)$ we have $\de(\pi(f)(g)) \circ f = f \circ \de(g)$ in $\L$.

A \emph{$p$-local finite group} is a triple $(S,\F,\L)$ of a saturated fusion system on $S$ together with an associated linking system.
It is not known if any saturated fusion system admits a centric linking system, and if so, if it is unique.

Now suppose that $\F=\F_S(\pi)$ for a group $\pi$ which contains $S$ as a Sylow $p$-subgroup.
A \emph{signaliser functor} on $\pi$ is an assignment $P \mapsto \Theta(P)$ for every $\F$-centric $P \leq S$ such that $\Theta(P)$ is a complement of $Z(P)$ in  $C_\pi(P)$ and such that if $gPg^{-1} \leq Q$ then $\Theta(Q) \leq g\Theta(P)g^{-1}$.
Indeed, $\Theta$ is a contravariant functor from $\T^c_S(\pi)$ to the category of groups.
A signaliser functor $\Theta$, if it exists, gives rise to a centric linking system $\L=\L_{\Theta}$ where $\L(P,Q)=N_\pi(P,Q)/\Theta(P)$.
See \cite{Asch-Chermak} and also \cite[Section 3]{Lib-Vir} for more details.

\subsection{Homotopy theoretic constructions}\label{subsec:sfs-homotopy-theoretic-constructions}

Fix a finite $p$-group $S$ and a map $f \colon BS \to X$.
By \cite[Definition 7.1]{BLO2} $f$ gives rise to 
a fusion system $\F_S(f)$ on $S$ whose objects are the subgroups of $S$ and a group monomorphism $\vp \colon P \to Q$ belongs to $\F_S(f)$ if and only if $BP \xto{B\vp} BS \xto{f} X$ is homotopic to $BP \xto{Bi} BS \xto{f} X$.
It follows directly from the definitions that any map $g \colon X \to Y$ induces an inclusion of fusion systems (not saturated in general)
\[
\F_S(f) \subseteq \F_S(g \circ f).
\]
The category $\L_S(f)$, see \cite[Definition 7.2]{BLO2}, has the same object set as $\F=\F_S(f)$.
The morphisms $P \to Q$ in $\L_S(f)$ are pairs $(\vp,[H])$ where $\vp \in \F(P,Q)$ and $[H]$ is the homotopy class of a path $H$ in $\map(BP,X)$ from $BP \xto{B\vp} BS \xto{f} X$ to $BP \xto{Bi} BS \xto{f} X$.
There is a surjective functor $\pi \colon \L_S(f) \to \F_S(f)$ which ``forgets'' $[H]$.

A map $g \colon X \to Y$ gives rise, by composition, to a functor 
\[
g_* \colon \L_S(f) \to \L_S(g \circ f)
\]
over the projection functors to $\F_S(f) \subseteq \F_S(g \circ f)$, i.e. $\pi \circ g_* = \incl_{\F_S(f)}^{\F_S(gf)} \circ \pi$.

Define $\L^c_S(f)$  as the full subcategory of $\L_S(f)$ on the $\F_S(f)$-centric objects.
If $g \colon X \to Y$ is a map then $\F_S(f) \subseteq \F_S(g \circ f)$ so every $\F_S(g \circ f)$-centric subgroup of $S$ is also $\F_S(f)$-centric and we obtain a functor
\[
g_* \colon \L^c_S(f) \to \L_S^c(g \circ f)
\]
defined on the $\F_S(g \circ f)$-centric objects of $\L^c_S(f)$.

\subsection{Classifying spaces}\label{subsec:classifying-spaces}
The \emph{classifying space} of a $p$-local finite group is $\pcomp{|\L|}$.
The inclusion $\de \colon S \to \Aut_\L(S)$ thought of as a full subcategory of $\L$ gives rise to a map 
\begin{equation}\label{eq:theta}
\theta \colon BS \to \pcomp{|\L|}
\end{equation}
which will play a fundamental role in this paper.
One of the key features of the map $\theta$ is \cite[Proposition 7.3]{BLO2}
\[
\F=\F_S(\theta) \qquad \text{ and } \qquad \L=\L^c_S(\theta).
\]
For any group $G$ let $\B G$ denote the category with one object $o_G$ and $G$ as its automorphism set.
A homomorphism $\vp \colon G' \to G$ induces a functor $\B\vp \colon \B G' \to \B G$.
It is well known that if $H \leq G$ then the functor $\B C_G(H) \times \B H \xto{(x,y) \mapsto xy} \B G$ gives rise, after realization and adjunction, to a homotopy equivalence $BC_G(H) \xto{\simeq} \map(BH,BH)_{Bi}$ where the space on the right is the path component of the inclusion $BH \xto{Bi} BG$.

\begin{prop}\label{prop:topological-linking-of-bpi}
Fix an inclusion $S \leq \pi$ of a Sylow $p$-subgroup and consider $Bi \colon BS \to B\pi$.
Then $\F_S(Bi)=\F_S(\pi)$ and $\L_S(Bi)=\T_S(\pi)$.
\end{prop}

\begin{proof}
The first isomorphism is clear because $[BP,B\pi]=\Hom(P,\pi)/\Inn(\pi)$.
Define a functor $\Phi \colon \T_S(\pi) \to \L_S(Bi)$ as follows.
It is the identity on objects and for a morphism $g \in N_\pi(P,Q)$ consider $c_g \colon P \to Q$ and the natural transformation $g \colon \B c_g \to \B i_P^\pi$ between the functors $c_g,i_P^S \colon \B P \to \B \pi$ given by $o_\pi \xto{g} o_\pi$.
This natural transformation gives rise to a homotopy $I \times BP \to B\pi$ from $Bc_g$ and $Bi_P^\pi$, whence to a path $H$ in $\map(BP,B\pi)$ from $Bc_g$ to $Bi_P^\pi$.
Define $\Phi(g)=(c_g,[H])$.

The functoriality of $\Phi$ is easy to check.
It is bijective because $\Phi$ maps the coset $C_\pi(P)g \leq N_\pi(P,Q)$ bijectively onto $\pi_1^\#(\map(BP,B\pi);Bc_g,Bi_P^\pi) \cong  \pi_1BC_\pi(P)$ via $\B C_\pi(P) \times \B P \to \B \pi$.
\end{proof}

\subsection{The subgroup decomposition}\label{subsec:sfs-subgroup-decomposition}

Recall that the orbit category $\O=\O(\F)$ of a fusion system $\F$ has the same object set as $\F$ and $\O(P,Q)=\F(P,Q)/\Inn(Q)$.
See \cite[Definition 2.1]{BLO2}.
Let $\O(\F^c)$ denote the full subcategory on the $\F$-centric objects.

Fix a $p$-local finite group $(S,\F,\L)$.
By \cite[Proposition 2.2]{BLO2} there is a functor called the \emph{subgroup decomposition}
\[
\be \colon \O(\F^c) \to \Top
\]
which as a functor into $\hotop$ - the homotopy category of spaces -  is isomorphic to the functor $P \mapsto BP$.
Moreover, there is a homotopy equivalence  
\[
\hocolim_{\O(\F^c)} \be \xto{\simeq} |\L|
\] 
whose restriction to $\be(S)\simeq BS$ is homotopic to the map $BS \to |\L|$ induced by the inclusion $S \leq \Aut_\L(S)$.
Hence $\pcomp{\theta}$ is homotopic to the map in display \eqref{eq:theta}.

\subsection{The normalizer decomposition}\label{subse:sfs-normaliser-decomposition}

Let $\R$ be a set of $\F$-centric subgroups of $S$ which is closed under $\F$-conjugation.
The subdivision category $\bar{S}(\R)$ is the following poset; see \cite{Lib06} for more details where it is denoted $\bar{s}d(\R)$.
The objects of $\bar{S}(\R)$ are the $\F$-conjugacy classes $[P_\bullet]$ of chains of inclusions $P_0 \lneq \dots \lneq P_k$ of subgroups in $\R$ called ``$k$-simplices'' ($k \geq 0$).
We say that $P_\bullet$ is $\F$-conjugate to $P_\bullet'=P_0' \lneq \dots \lneq P_k'$ if there is an isomorphism $P_k \to P_k'$ in $\F$ which carries $P_i$ isomorphically onto $P_i'$ for all $i$.
There is a unique morphism $[P_\bullet] \to [Q_\bullet]$ in $\bar{S}(\R)$ if $Q_\bullet$ is $\F$-conjugate to a sub-sequence (a ``subsimplex'') of $P_\bullet$.

Let $\Aut_\F(P_\bullet)$ be the subgroup of $\Aut_\F(P_k)$ which leaves $P_0,\dots,P_k$ invariant.
Let $\Aut_\L(P_\bullet)$ be its preimage in $\Aut_\L(P_k)$.
Suppose that $Q_\bullet=Q_0\lneq \dots \lneq Q_m$ is a sub-simplex of $P_\bullet$ where $Q_m=P_i$ for some $i$.
Consider $e \in N_\pi(P_i,P_k)$ and set $\hat{e}=\de(e)\in \L(P_i,P_k)$.
For every $\vp \in\Aut_\L(P_\bullet)$ we have $\pi(\vp) \in \Aut_\F(P_\bullet)$ and therefore $\pi(\vp)|_{P_i} \in \Aut_\F(Q_\bullet)$.
It follows from \cite[Lemma 1.10(a)]{BLO2} and the fact that $\hat{e}$ is both an epimorphism and a monomorphism in $\L$, see \cite[Corollary 3.10]{BCGLO} that there exists a unique $\vp' \in \Aut_\L(Q_\bullet)$ such that $\vp \circ \hat{e}=\hat{e} \circ \vp'$.
This sets up a canonical group monomorphism $\Aut_\L(P_\bullet) \hookrightarrow \Aut_\L(Q_\bullet)$.
See \cite{Lib06} for more details.
We obtain a functor $[P_\bullet] \mapsto B\Aut_\L(P_\bullet)$ into the homotopy category of spaces (here we use the fact that the self-map that an inner automorphism of a group induces on its classifying space is homotopic to the identity).

The \emph{normalizer decomposition} with respect to the collection $\R$ of $\F$-centric subgroups, see \cite[Theorem A]{Lib06}, is a functor 
\[
\de \colon \bar{S}(\R) \to \Top
\]
which as a functor into $\hotop$ is isomorphic to the functor $[P_\bullet] \mapsto B\Aut_\L(P_\bullet)$.
In addition there is a map
\[
\hocolim_{\bar{S}(\R)} \de \xto{g} |\L|
\]
 whose restriction to $\de([P_\bullet])$ is induced, upon taking nerves, from the inclusion of categories $\B\Aut_\L(P_\bullet) \subseteq \L$.
Hence, the restriction of $\hocolim_{\bar{S}(\R)} \de \simeq |\L| \to \pcomp{|\L|}$ to $BS\leq B\Aut_\L(S)$ is homotopic to $\theta$ in display \eqref{eq:theta}.
If $\R$ contains all the $\F$-centric $\F$-radical subgroups of $S$ then $g$ is a homotopy equivalence, \cite[Theorem A]{Lib06}.

\begin{proof}[Proof of Theorem \ref{thmA}]
\ref{thmA:induce-fusion}
Let $Bi \colon BS \to B\pi$  denote the inclusion.
By Proposition \ref{prop:topological-linking-of-bpi} the map $f$ induces an inclusion
\[
\F_S(\pi) = \F_S(Bi) \subseteq \F_S(f \circ Bi) = \F_S(\theta) = \F.
\]
By hypothesis $\F \subseteq \F_S(\pi)$ and therefore equality holds.

\ref{thmA:signaliser}
We deduce from Proposition \ref{prop:topological-linking-of-bpi} and from \S\ref{subsec:sfs-homotopy-theoretic-constructions} that $\L^c_S(Bi)=\T^c_S(\pi)$, that $\L_S^c(f \circ Bi)=\L$  because $f \circ Bi \simeq \theta$, and that $f$ induces a functor $\rho \colon \T^c_S(\pi) \to \L$ such that
\[
\xymatrix{
{\T^c_S(\pi)} \ar@{->>}[d]_{\pi} \ar[r]^\rho &
\L \ar@{->>}[d]^\pi 
\\
{\F^c} \ar@{=}[r] & 
\F^c
}
\]
commutes.
We will now prove, and this is the key observation, that $\rho$ is surjective.

Set $\T=\T^c_S(\pi)$.
Since $Z(P) \leq \Aut_\L(P)$ acts on $\L(P,Q)$ with orbit set $\F(P,Q)$ and since the functors $\pi$ in the square above are surjective, it suffices to prove that $Z(P)$ is in the image of $\rho_P \colon \Aut_\T(P) \to \Aut_\L(P)$ for any $\F$-centric $P \leq S$.
To do this consider the inclusion $Bi_P^S \colon BP \to BS$.
The composition
\[
\map(BP,BS)_{Bi_P^S} \xto{Bi_*} \map(BP,B\pi)_{Bi_P^\pi} \xto{f_*} \map(BP,\pcomp{|\L|})_{f \circ Bi_P^\pi}
\]
is a homotopy equivalence because $P$ is $\F$-centric so the first space is homotopy equivalent to $BC_S(P)=BZ(P)$ and because the third space is homotopy equivalent to $BZ(P)$ by \cite[Theorem 4.4(c)]{BLO2} since $Bi \circ f \simeq \theta$.
We see from the definition of $\L_S(-)$, see \S\ref{subsec:sfs-homotopy-theoretic-constructions},  that $\rho_P$ carries $Z(P) \leq N_\pi(P)$ isomorphically onto $Z(P) \leq \Aut_\L(P)$.

We will now define the signaliser functor $\Theta$.
For every $\F$-centric $P \leq S$ set
\[
\Theta(P) \overset{\text{def}}{=} \ker ( \Aut_\T(P) \to \Aut_\L(P) ).
\]
Now, $\Aut_\T(P)=N_\pi(P)$ so we obtain a commutative diagram with exact rows
\[
\xymatrix{
1 \ar[r] &
C_\pi(P) \ar[d] \ar[r] &
N_\pi(P) \ar@{->>}[d]^{\rho_P} \ar[r]^\pi &
\Aut_\F(P) \ar[r] \ar@{=}[d] &
1
\\
1 \ar[r] &
Z(P) \ar[r] &
\Aut_\L(P) \ar[r]_\pi &
\Aut_\F(P) \ar[r] &
1
}
\]
which implies that $\Theta(P) \leq C_\pi(P)$ and that it is a complement of $Z(P)$ because we have seen that $\rho_P(Z(P)=Z(P)$.
For $g \in \pi$ such that $gPg^{-1} \leq S$ we obtain a commutative diagram
\[
\xymatrix{
1 \ar[r] &
\Theta(P) \ar[r] \ar[d] &
N_\pi(P) \ar[r] \ar[d]^{x \mapsto gxg^{-1}}_\cong &
\Aut_\L(P) \ar[r] \ar[d]^{\vp \mapsto \rho(g)\vp\rho(g)^{-1}}_\cong &
1
\\
1 \ar[r] &
\Theta(gPg^{-1}) \ar[r] &
N_\pi(gPg^{-1}) \ar[r] &
\Aut_\L(gPg^{-1}) \ar[r] &
1
}
\]
which shows that $g\,\Theta(P)\, g^{-1} = \Theta(gPg^{-1})$.

Now suppose that $P \leq Q$ and set $N=N_\pi(P) \cap N_\pi(Q)$.
Set $\hat{e}=\de(e) \in \Aut_\L(P,Q)$ where $e \in N_S(P,Q)$ is the identity element.
Let $\Aut_\L(Q,\darrow_P)$ be the subgroup of $\Aut_\L(Q)$ of the automorphisms $\vp$ which restrict to an automorphism of $P$ in the sense that there exists $\vp' \in \Aut_\L(P)$ such that $\vp \circ \hat{e} = \hat{e} \circ \vp'$.
Note that $\vp'$ is unique because $\hat{e}$ is a monomorphism in $\L$ by \cite[Corollary 3.10]{BCGLO}.
Similarly, let $\Aut_\L(P,\uparrow^Q)$ denote the subgroup of $\Aut_\L(P)$ of the automorphisms $\vp'$ which extend to an automorphism of $Q$ in the sense that there exists $\vp \in \Aut_\L(Q)$ such that $\vp \circ \hat{e} = \hat{e} \circ \vp'$.
Note that $\vp$ is unique because $\hat{e}$ is an epimorphism in $\L$, \cite[Corollary 3.10]{BCGLO}.
Thus $\vp \mapsto \vp'$ sets up an isomorphism $\Aut_\L(Q,\darrow_P) \to \Aut_\L(P,\uparrow^Q)$.
Note that if $\vp = \rho_Q(g)$ for some $g \in N_\pi(Q)$ then $g \in N$ and $\vp'=\rho_P(g)$.

Clearly, the preimage under $\rho$ of $\Aut_\L(Q,\darrow_P)$ in $\Aut_\T(Q)=N_\pi(Q)$ is $N$ as is the preimage of $\Aut_\L(P,\uparrow^Q)$ in $\Aut_\T(P)=N_\pi(P)$.
Also, $C_\pi(Q) \leq N$.
We obtain 
\[
\xymatrix{
1  \ar[r] &
\Theta(Q) \ar[r] \ar[d] &
N \ar[r]^(0.3){\rho_Q} \ar@{=}[d] &
\Aut_\L(Q,\darrow_P) \ar[d]_\cong^{\vp \mapsto \vp'} \ar[r] &
1
\\
1 \ar[r] &
\Theta(P) \cap N \ar[r] &
N \ar[r]_(0.3){\rho_P} &
\Aut_\L(P,\uparrow^Q) 
\ar[r] &
1
}
\]
We deduce that $\Theta(Q) =\Theta(P) \cap N \leq \Theta(P)$, hence $\Theta$ is a signaliser functor.

\ref{thmA:retract}
Fix some $\vp \in \F(P,Q)$.
There is some $g \in \pi$ such that $\vp$ is the restriction of $c_g \in \Inn(\pi)$ to $P$ because $\F = \F_S(\pi)$.
Since $Bc_g$ is homotopic to the identity on $B\pi$, then following square commutes up to homotopy
\[
\xymatrix{
BP \ar[r]^{Bi_P^S} \ar[d]_{B\vp} &
BS \ar[r]^{Bi} & 
B \pi \ar@{=}[d]^{Bc_g \simeq \id}
\\
BQ \ar[r]_{Bi_P^S} & 
BS \ar[r]_{Bi} &
B\pi.
}
\]
Applying $H^*(-;\FF_p)$ to this diagram for every $\vp \in \F^c$ shows that that the image of $H^*(B\pi;\FF_p) \xto{Bi^*} H^*(BS;\FF_p)$ is contained in the ring of stable elements $H^*(\F^c;\FF_p)$, namely the subring of all the classes $\al \in H^*(S;\FF_p)$ such that $\vp^*(\al|_Q)=\al|_P$ for all $\vp \in \F^c(P,Q)$. 

Now we consider $BS \xto{Bi} B\pi \xto{f} \pcomp{|\L|}$ which is homotopic to $\theta$.
By \cite[Theorem 5.8]{BLO2} $\theta^*$ is an isomorphism of $H^*(\pcomp{|\L|};\FF_p)$ onto $H^*(\F^c) \leq H^*(BS;\FF_p)$.
But $\theta^*=(Bi)^* \circ f^*$ so $Bi^*$ is left inverse of $f^*$, both are $\FF_p$-algebra homomorphisms.
\end{proof}


Ragnarsson defined in \cite{Ragnarsson} a classifying spectrum $\BB\F$ for any saturated fusion system $\F$.
It is equipped with a structure map $\si_\F \colon \BB S \to \BB\F$ where $\BB S=\Sigma^\infty BS$.
In the presence of a linking system, $\BB\F \simeq \Sigma^\infty \pcomp{|\L|}$ and $\si_\F=\Sigma^\infty \theta$, see display \eqref{eq:theta}.
He showed that the assignment of $\F \mapsto \BB\F$ is a functor from the category of fusions systems to the homotopy category of spectra.
In addition he showed that $\F$ can be recovered from $(\BB \F,\si_\F)$ in the sense that the morphisms in $\F$ are exactly the group monomorphisms $\vp \colon P \to Q$ such that $\BB P \xto{\BB \vp} \BB S \xto{\si_\F} \BB \F$ is homotopic to $\BB P \xto{\BB i} \BB S \xto{\si_\F} \BB \F$ where $\BB P = \Sigma^\infty BP$.
Compare this with the construction in \S\ref{subsec:sfs-homotopy-theoretic-constructions}.

Here is a variant of Theorem \ref{thmA}\ref{thmA:induce-fusion}.
Note that while linking systems are not known to exists, the classifying spectra of saturated fusion systems always exist.

\begin{prop}\label{prop:stable-fusion}
Suppose $\F$ is a saturated fusion system on $S$ and that $S$ is a Sylow $p$-subgroup of some group $\pi$.
Assume that there exists a map $f \colon B\pi \to \Omega^\infty \BB \F$ whose restriction to $BS$ is homotopic to the right adjoint $\si_\F^\# \colon BS \to \Omega^\infty \BB \F$ of $\si_\F \colon \BB S \to \BB\F$.
Then $\F_S(\pi) \subseteq \F$.
\end{prop}

\begin{proof}
By Proposition \ref{prop:topological-linking-of-bpi} and the constructions in \S\ref{subsec:sfs-homotopy-theoretic-constructions}, it only remains to prove that $\F=\F_S(\si_\F^\#)$.
This is straightforward from the definition of $\F_S(\si_\F^\#)$ and the adjunction $[ \BB P , \BB \F] \cong [\Sigma^\infty BP , \BB \F] \cong [BP,\Omega^\infty \BB \F]$.
\end{proof}

%
%
%
\section{Graphs and homotopy colimits}

\subsection{Categories of dimension $1$}\label{ss:cat-dim-one}

A small category $\C$ is called $1$-dimensional if its nerve is $1$-dimensional.
Thus, in any pair of composable morphisms in $\C$ at least one of them is an identity morphism.
Equivalently, there is a partition $\Obj(\C)=C_1 \coprod C_2$ such that if $\vp \colon x \to y$ is a non-identity morphism in $\C$ then $x \in C_1$ and $y \in C_2$.

Since there are no non-trivial compositions in $\C$,  it is clear that if $\D$ is any category then any assignment $F \colon \Obj(\C) \to \Obj(\D)$ together with an assignment of a morphism $F(c_1) \to F(c_2)$ to any \emph{non-identity} morphism $c_1 \to c_2$ in $\C$, gives rise to a functor $F \colon \C \to \D$.

\begin{prop}\label{prop:1-diml-htpy-nat-maps}
Let $\C$ be a $1$-dimensional category and let $F,F' \colon \C \to \Top$ be functors whose values are CW-complexes.
Assume that there is an isomorphism $q \colon F \xto{\simeq} F'$ in the \emph{homotopy category of spaces $\hotop$}.
Then there is a homotopy equivalence $\bar{q} \colon \hocolim F \xto{\simeq} \hocolim F'$ with the property that the squares 
\[
\xymatrix{
F(c) \ar[r] \ar[d]_{q_c}^{\simeq} &
\hocolim F \ar[d]^{\bar{q}}_{\simeq}
\\
F'(c) \ar[r]  &
\hocolim F'
}
\] 
commute up to homotopy for any $c \in \C$.
\end{prop}

\begin{proof}
Let $C_1 \coprod C_2$ be a partition of $\Obj(\C)$ as described above.
The isomorphism $q \colon F \to F'$ in $\hotop$ gives rise to the following maps which we will fix once and for all.
Homotopy equivalences $q_c \colon F(c) \xto{\simeq} F'(c)$ for every object $c \in \C$ and, for every non-identity morphism $\vp \colon c_1 \to c_2$ in $\C$, a homotopy $h_\vp \colon F(c_1) \times I \to F'(c_2)$ from $q_{c_2} \circ F(\vp)$ to $F'(\vp) \circ q_{c_1}$.

Define functors $H, H_0 , H_1 \colon \C \to \Top$ by describing their effect on objects of $C_1$ and $C_2$ and on non-identity morphisms as follows
\[
\begin{array}{lll}
H_0 \colon c_1 \mapsto F(c_1), &
H_0 \colon c_2 \mapsto F'(c_2), &
H_0(c_1 \xto{\vp} c_2) = q_{c_2} \circ F(\vp).
\\
H_1 \colon c_1 \mapsto F(c_1), &
H_1 \colon c_2 \mapsto F'(c_2), &
H_1(c_1 \xto{\vp} c_2) = F'(\vp) \circ q_{c_1}
\\
H \colon c_1 \mapsto F(c_1) \times I, &
H \colon c_2 \mapsto F'(c_2), &
H(c_1 \xto{\vp} c_2) = h_\vp.
\end{array}
\]
Now we will define natural transformations, all of which are homotopy equivalences:
\[
F \xrightarrow[\simeq]{(\id,q_2)} H_0 \xrightarrow[\simeq]{(i_0,\id)} H \xleftarrow[\simeq]{(i_1,\id)} H_1 \xrightarrow[\simeq]{(q_1,\id)} F'
\]
On objects of $C_1$ and $C_2$ the effect of these natural transformations is described by the rows of the diagram below and the columns show the effect of $F, H_0, H, H_1$ and $F'$ on a non-identity morphism $\vp$.
The commutativity of all the squares shows that these are indeed natural transformations.
\[
\xymatrix{
F(c_1) \ar[r]^{\id} \ar[d]_{F(\vp)} &
F(c_1) \ar[r]^{i_0} \ar[d]^{q_{c_2}\circ F(\vp)} &
F(c_1) \times I \ar[d]^{h_\vp} &
F(c_1) \ar[l]_{i_1} \ar[r]^{q_{c_1}} \ar[d]^{F'(\vp) \circ q_{c_1}} &
F'(c_1) \ar[d]^{F'(\vp)}
\\
F(c_2) \ar[r]^\simeq_{q_{c_2}} &
F'(c_2) \ar[r]_{\id} &
F'(c_2) &
F'(c_2) \ar[l]^{\id} \ar[r]_{\id} &
F'(c_2)
}
\]
We now obtain  homotopy equivalences
\[
\hhocolim{\C}F \xto{\simeq} \hhocolim{\C} H_0 \xto{\simeq} \hhocolim{\C} H 
\xleftarrow{\simeq} \hhocolim{\C} H_1 \xto{\simeq} \hhocolim{\C} F'.
\]
It is also clear from the commutative ladder above that the second diagram in the statement of this proposition is homotopy commutative for all $c \in \C$.
\end{proof}

\subsection{Graphs and their associated categories}
\label{subsec:graphs-associated-categories}

A special class of $1$-dimensional categories is associated to graphs.
Recall from Serre \cite{Serre:trees} that a graph $Y$ consists of the following data.
A set $V$ of vertices also denoted $\vertx(Y)$ and a set $E$ of (directed) edges denoted $\edge(Y)$.
There is a function $E \xto{y \mapsto (o(y),t(y)) } V \times  V$ whose components are the \emph{origin} and \emph{terminus} of an edge $y$.
It is also equipped with an involution $E \xto{y \mapsto \bar{y}} E$ which has no fixed points and ``reverses the  direction'' of an edge, namely $o(\bar{y})=t(y)$ and $t(\bar{y})=o(y)$.
The orbits of this involution $[y]$ are called ``geometric edges''.
The set of the geometric edges of $Y$ is denoted $[\edge(Y)]$.

A \emph{walk} in $Y$ from a vertex $u$ to a vertex $v$ is a sequence $y_1,\dots,y_n$ of edges such that $o(y_1)=u$ and $t(y_n)=v$ and $o(y_{i+1})=t(y_i)$ for all $i$.
A \emph{loop} is a walk from a vertex to itself.
A \emph{path} from $u$ to $v$ is a walk which contains no loops.

An \emph{orientation} of $Y$ is a choice $A\subseteq E$ of one representative from each geometric edge.
Given an orientation $A$ we will write $|y|$ for either the edge $y$ or $\bar{y}$ which belongs to $A$.

Any graph $Y$ gives rise to a small category $\C_Y$ whose object set is the disjoint union of $\vertx(Y)$ and $[\edge(Y)]$.
Apart from identity morphisms there is a unique morphism $[y] \xto{y} t(y)$ for any $y \in \edge(Y)$.

\subsection{Graphs of groups}

Let $Y$ be a graph.
Recall from \cite[Def. 8 in \S4.4]{Serre:trees} that a graph of groups $(G,Y)$ is an assignment of groups $G_v$ to every $v \in \vertx(Y)$ and $G_y$ to every $y \in \edge(Y)$ such that $G_y=G_{\bar{y}}$ and every edge is associated with a group monomorphism $G_y \to G_{t(y)}$ which is denoted $a \mapsto a^y$ (here $a \in G_y$ and $a^y$ is its image in $G_{t(y)}$, \cite[pp. 41--42]{Serre:trees}).

We obtain an associated functor $\G \colon \C_Y \to \{\Gps\}$ defined by
\[
\G(v)=G_v, \qquad \G([y])=G_{y}, \qquad \G([y] \xto{y} t(y)) = G_y \xto{a \mapsto a^y} G_{t(y)}.
\]
Clearly, $\G$ carries morphisms in $\C_Y$ to group monomorphisms.
Conversely, any functor $\G \colon \C_Y \to \{\Gps\}$ with this property defines a graph or groups $(G,Y)$.

\subsection{The fundamental group}\label{subsec:universal-cover}

Fix a graph of groups $(G,Y)$. 
Let $F(\edge(Y))$ denote the free group whose generating set is $\edge(Y)$ and set 
\[
\tilde{F}(G,Y) = F(\edge(Y)) * \big( \underset{v \in \vertx Y}{*} G_v \big).
\]
Define $F(G,Y)$ as the quotient group
\[
F(G,Y) = \tilde{F}(G,Y) \Big/ \langle \bar{y} = y^{-1}, ya^y y^{-1}=a^{\bar{y}} \rangle
\]
for all $y \in \edge(Y)$ and all $a \in G_y=G_{\bar{y}}$.
Let $T$ be a maximal tree in $Y$.
The \emph{fundamental group} of $(G,Y)$ is defined as the factor group, see \cite[p. 42]{Serre:trees}
\[
\pi_1(G,Y,T)=F(G,Y)/\langle y=1 \colon y \in T \rangle.
\]
The image of $y$ in this group is denoted $g_y$.
It is shown in \cite[p. 44]{Serre:trees} that different choices of $T$ yield isomorphic groups.
By \cite[Cor. 1, p. 45]{Serre:trees}, the inclusions $G_v \leq F(G,Y)$,  for any $v \in \vertx(Y)$, induce injections 
\[
G_v \hookrightarrow \pi_1(G,Y,T).
\]
It also follows from the definitions that for any $y \in \edge Y$ and for any $a \in G_y$ 
\[
g_y \cdot a^y \cdot g_y{}^{-1} = a^{\bar{y}}.
\]
Thus $g_y$ conjugates the image of $G_y \xto{a \mapsto a^y} G_{t(y)}$ onto the image of $G_y \xto{a \mapsto a^{\bar{y}}} G_{o(y)}$.

\subsection{Homotopy colimits and graphs of groups}\label{subsec:hocolim-groups}

By applying the classifying-space functor $B(-)$  to  the functor $\G \colon \C_Y \to \{\Gps\}$ associated to $(G,Y)$ we obtain a functor $B\G \colon \C_Y \to \Top$.
The following result was observed by Dror-Farjoun in \cite[\S 4.1]{DF04}.

\begin{prop}
Consider $\G\colon \C_Y \to \Top$ associated  with  $(G,Y)$.
Then
\[
\hocolim_{\C_Y} B\G \simeq B\pi
\]
where $\pi=\pi_1(G,Y,T)$ for some choice of a maximal tree $T$ and the maps $BG_v \to \hocolim_{\C_Y} B\G \simeq B\pi$ are homotopic to $B(\incl_{G_v}^\pi)$. 
\end{prop}

\begin{prop}\label{prop:sylows}
Let $(G,Y)$ be a graph of groups and suppose that
\begin{itemize}
\item[(i)]
The groups $G_v$ and $G_y$ contain Sylow $p$-subgroups $P_v$ and $P_y$ for every $v \in \vertx Y$ and every $y \in \edge Y$.

\item[(ii)]
There exists a vertex $v_0$ such that for any other vertex $u$ of $Y$ there exists a path (directed, without loops) $y_1,\dots,y_n$ from $v_0$ to $u$ such that for any $i$ the map $G_{y_i} \xto{a \mapsto a^{y_i}} G_{t(y_i)}$ carries $P_{y_i}$ onto a Sylow $p$-subgroup of $G_{t(y_i)}$.
\end{itemize}
Set $\pi=\pi_1(G,Y,T)$ 
and set $S=P_{v_0}$.
Then $\hocolim_{\C_Y} B\G \simeq B\pi$ and
\begin{enumerate}
\item
$S$ is a Sylow $p$-subgroup of $\pi$.
\label{prop:sylows:exist}

\item
If the groups $G_v$ are finite and if $Y$ contains no loops then 
\label{prop:sylows:finite}
\begin{enumerate}
\item
 $B\pi$ is a $p$-good space (\cite[I.\S 5]{Bous-Kan}), and 
\label{prop:sylows:finite:good}

\item
In the kernel $K$ of $H^*(B\pi;\FF_p) \xto{Bi^*} H^*(BS;\FF_p)$ the product of any two elements is zero. 
\label{prop:sylows:finite:ker}
\end{enumerate}
\end{enumerate}
\end{prop}

\begin{proof}
\ref{prop:sylows:exist}.
By \cite[I.\S 5.3]{Serre:trees} the group $\pi$ acts without inversions on a tree $\tilde{X}$ such that the isotropy groups of the vertices of $\tX$ are conjugate to the subgroup $G_v$ of $\pi$ where $v \in \vertx Y$.
If $P$ is a finite $p$-subgroup of $\pi$ then by \cite[Prop. 19, \S I.4.3]{Serre:trees} it fixes a vertex of $\tX$ and it is therefore conjugate in $\pi$ to a subgroup of $G_v$ for some vertex $v$.
We therefore assume that $P \leq G_v$.

By hypothesis there is a a path $y_1,\dots,y_n$ from $v_0$ to $v$ with the properties listed in (ii).
We will use induction on $n$ to prove that $P$ is conjugate to a subgroup of $S$.
If $n=0$ then $P \leq G_{v_0}$ and the result holds because $S$ is a Sylow $p$-subgroup of $G_{v_0}$.
Suppose that $n >0$ and note that by hypothesis the image of $G_{y_n}$ in $G_{t(y_n)}=G_v$, which we denote by $G_{y_n}^{y_n}$, contains a Sylow $p$-subgroup.
Thus, by conjugating $P$ inside $G_v$ we may assume that $P \leq G_{y_n}^{y_n}$.
As we have seen in \S\ref{subsec:universal-cover} the element $g_{y_n} \in \pi$ conjugates $G_{y_n}^{y_n}$ onto $G_{y_n}^{\bar{y_n}} \leq G_{o(y_n)} = G_{t(y_{n-1})}$.
We deduce that $P$ is conjugate in $\pi$ to a subgroup of $G_{t(y_{n-1})}$ and we can now apply the induction hypothesis to the path $y_1,\dots,y_{n-1}$ and deduce that $P$ is conjugate in $\pi$ to a subgroup of $S$.

\ref{prop:sylows:finite}
Since $Y$ has no loops, it is a tree.
Orient the tree $Y$ in the only way so that there is a path from $v_0$ to any vertex of $Y$ which only uses edges from the orientation class $A$.
We now claim that by possibly replacing $(G,Y)$ with an isomorphic graph, the groups $P_v$ and $P_y$ can be chosen such that $(P,Y)$ is a sub-graph of $(G,Y)$, namely $G_y \xto{a \mapsto a^y} G_{t(y)}$ carries $P_y$ into $P_{t(y)}$.
In fact, by hypothesis (ii) $P_y \cong P_{t(y)}$ if $y \in A$.

To prove this, consider the poset $\T$ of the subtrees $Y'$ of $Y$ with root $v_0$ for which is it possible to replace $(G,Y)$ with an isomorphic tree of groups and to find Sylow $p$-subgroup $P_v \leq G_v$ and $P_y \leq G_y$ such that $(P,Y')$ is a subtree of $(G,Y')$.
First, $\T$ is not empty because clearly $\{v_0\} \in \T$.
Next, consider a maximal element $Y' \in \T$ and assume that $Y' \neq Y$.
Thus, we may assume, possibly by replacing $(G,Y)$ with an isomorphic tree, that $(P,Y')$ is a subtree of $(G,Y')$.
Choose an edge $y' \in A$ such that $v'=o(y') \in Y'$ and $v=t(y') \notin Y'$.

Let $G_{\bar{y'}}^{\bar{y'}}$ denote the image of $G_{y'}$ in $G_{o(y')}$ and let $P_{\bar{y'}}^{\bar{y'}}$ be the image of $P_{y'}$.
Then $g P_{\bar{y'}}^{\bar{y'}} g^{-1} \leq P_{o(y')}$ for some $g \in G_{o(y')}$ by the Sylow condition.
Let $c_g$ denote the inner automorphism $x \mapsto gxg^{-1}$.
By replacing $G_{y'}$ with $g G_{\bar{y'}}^{\bar{y'}} g^{-1}$ we obtain an isomorph $(\tilde{G},Y)$ of $(G,Y)$.
Now replace $P_{y'}$ with $g P_{\bar{y'}}^{\bar{y'}} g^{-1}$ and replace $P_{t(y')}$ with the image of $P_{y'}$ in $G_{t(y')}=G_v$ which is a Sylow $p$-subgroup by hypothesis (ii).
Note that $o(y') \neq t(y')$ so we have left the trees $(P,Y')$ and $(G,Y')$ unchanged.
By construction $(P,Y' \cup \{v\})$ is a subtree of $(\tilde{G},Y'\cup\{v\})$ which contradicts the maximality of $Y'$.
Thus, $Y \in \T$ which is what we wanted to prove.

\ref{prop:sylows:finite:good}.
Since $Y$ is a tree, $\pi$ is the amalgamated product of finite groups which is generated by the groups $G_v$.
Also notice that the images of $P_v$ in $\pi$  are subgroups of $S$ because $P_y \to P_t(y)$ are isomorphisms for $y \in A$ (consider all that rooted subtrees $Y' \subseteq Y$ for which $P_v \leq S$ and argue as above).
Let $K$ be the normal subgroup of $\pi$ generated by the elements of order prime to $p$.
Clearly $K$ contains $O^p(G_v)$ and therefore together with the groups $P_v$ it generates $\pi$.
Hence $\pi=KS$, so $\pi/K$ is a finite $p$-group. 

Now $BK$ is $p$-good by \cite[Proposition VII.3.2]{Bous-Kan} since $K$ is $p$-perfect \cite[VII.3.2]{Bous-Kan} because its abelianisation must be a torsion group with no elements of order $p$.
The Bousfield-Kan spectral sequence $E^2_{ij}=\varinjlim{}_i H_j(B\G;\FF_p) \Rightarrow H_{i+j}(B\pi;\FF_p)$ where $\G$ is the diagram of groups associated to $(G,Y)$, clearly consists of finitely generated $\FF_p$-modules and therefore $H_*(B\pi;\FF_p)$ is finite in every dimension.
The action of $\pi/K$ on $H^*(BK;\FF_p)$ must be nilpotent by \cite[II.5.2]{Bous-Kan} so by the strong convergence of the Eilenberg-Moore spectral sequence for the fibration $BK \to B\pi \to B\pi/K$ proven by Dwyer in \cite{Dwyer73}, we deduce that $H_*(BK;\FF_p)$ is finite in every dimension.
By \cite[Chapter II, 5.1 and 5.2]{Bous-Kan} we obtain a nilpotent fibration $\pcomp{BK} \to \pcomp{B\pi} \to B\pi/K$ in which both the base and the fibre are $p$-complete, whence so is the total space.
This shows that $B\pi$ is $p$-good.

\ref{prop:sylows:finite:ker}.
Let $\G, \P \colon \C_Y \to \Gps$ be the associated functors for $(G,Y)$ and $(P,Y)$.
Observe that $\hocolim_{\C_Y} B\P = BS$ because $P_y \to P_{t(y)}$ are isomorphisms for all $y \in A$ (consider the poset $\T$ of rooted subtrees $Y' \subseteq Y$ for which $\hocolim B\P|_{\C_{Y'}} \simeq BS$ and argue as above).
We obtain a natural transformation of functors $B\P \to B\G$ which gives rise to a morphism of Bousfield-Kan spectral sequences $E^{*,*}_2(\G) \to E^{*,*}_2(\P)$ where
\begin{eqnarray*}
E^{i,j}_2(\P) &=& \varprojlim_{\C_Y}{}^i H^j(B\P;\FF_p) \Rightarrow H^{i+j}(BS;\FF_p), \\
E^{i,j}_2(\G) &=& \varprojlim_{\C_Y}{}^i H^j(B\G;\FF_p) \Rightarrow H^{i+j}(B\pi;\FF_p)).
\end{eqnarray*}
Since $P_v \leq G_v$ and $P_y \leq G_y$ are Sylow $p$-subgroups and $G_v$ are finite groups, the maps $H^*(G_v;\FF_p) \to H^*(P_v;\FF_p)$ and $H^*(G_y;\FF_p) \to H^*(P_y;\FF_p)$ are injective.
Since $\varprojlim{}^0$ is left exact, $E_2^{0,*}(\G) \to E_2^{0,*}(\P)$ is injective.
Also $E_2^{i,*}(\P)=E_2^{i,*}(\G)=0$ for all $i \geq 2$ because $\C_Y$ is $1$-dimensional.
Therefore $E_\infty^{**}=E_2^{**}$.
We obtain the following diagram with exact rows from which it is clear that $K=\ker(Bi^*)$ is contained in $E_\infty^{1,*}(\G)$.
\[
\xymatrix{
0 \ar[r] & 
E^{1,*-1}_2(\G) \ar[r] \ar[d] &
H^*(B\pi) \ar[r] \ar[d]_{Bi^*} &
E_2^{0,*}(\G) \ar@{^(->}[d] \ar[r] & 
0
\\
0 \ar[r] &
E_2^{0,*-1}(\P) \ar[r] &
H^*(BS) 
\ar[r] &
E_2^{0,*}(\P) \ar[r] &
0
}
\]
It follows that $xy=0$ for any $x,y \in K$ because $xy \in E^{2,*}_\infty(\G)=0$.
\end{proof}

%
%
%
\section{Groups inducing fusion systems}\label{sec:examples}

\subsection{Robinson's construction - the normalizer decomposition}\label{subsec:robinsons-construction}

Fix a $p$-local finite group $(S,\F,\L)$ and a collection $\R$ of $\F$-centric subgroups which contains all the $\F$-centric $\F$-radical subgroups (e.g. all the $\F$-centric subgroups of $S$).
Consider the normalizer decomposition $\de \colon \bar{S}(\R) \to \Top$, see \S\ref{subsec:sfs-homotopy-theoretic-constructions}, and recall that in the homotopy category of spaces, $\de$ is isomorphic to the functor $[R_\bullet] \mapsto B\Aut_\L(R_\bullet)$.

Fix fully normalized representatives $R_0,R_1,\dots,R_n$ for the $\F$-conjugacy classes of $\R$ and choose for convenience $R_0=S$.
Let $\D$ be the subposet of $\bar{S}(\R)$ generated by the objects $[R_0],[R_1],\dots,[R_n]$ and $[R_1<S], \dots,[R_n<S]$.
It is easy to check that $\D=\C_Y$ where $Y$ is the tree of height $1$ with root $[S]$ and edges $[R_i<S]$ to the leaves $[R_i]$ ($i=1,\dots n$).
We obtain a tree of groups $(G,Y)$, see \S\ref{subse:sfs-normaliser-decomposition}
\begin{equation}\label{eq:robinson-type-amalgam}
\xymatrix{
  & 
\Aut_\L(S) \ar@{-}[dl]_{\Aut_\L(R_1<S)} \ar@{-}[d]  \ar@{-}[dr]^{\Aut_\L(R_n<S)}
\\
{\Aut_\L(R_1)} \ar@{.}[r] & \Aut_\L(R_i) \ar@{.}[r] & \Aut_\L(R_n)
}
\end{equation}
This tree of groups satisfies the conditions of Proposition \ref{prop:sylows} because the groups $R_i$ are fully normalized so $N_S(R_i)$ is a Sylow $p$-subgroup of both $\Aut_\L(R_i)$ and $\Aut_\L(R_i<S) = N_{\Aut_\L(S)}(R_i)$.
Also note that $S$ is a Sylow $p$-subgroup of $\Aut_\L(S)$.
Therefore $\hocolim B\G \simeq B\pi_R$ for a group $\pi_R$ which is the amalgamated product of the tree \eqref{eq:robinson-type-amalgam} and which contains $S$ as a Sylow $p$-subgroup.

The crux, now, is that $B\G$ is the functor $[R_\bullet] \mapsto B\Aut_\L(R_\bullet)$ and therefore it is isomorphic to $\de|_\D$ in $\hotop$.
Proposition \ref{prop:1-diml-htpy-nat-maps} now yields a composite map $f$
\[
B\pi_R \simeq \hocolim_\D B\G \xto{\simeq} \hocolim_{\D} \de|_\D \xto{\incl} \hocolim_{\bar{S}(\R)} \de \xto{\simeq} |\L| \to \pcomp{|\L| }
\]
whose restriction to $BS$ is homotopic to $\theta$ in display \eqref{eq:theta}.
We now claim that $\F \subseteq \F_S(\pi_R)$.

\begin{proof}[Proof of claim]
First, $\F_S(\pi_R)$ contains all the groups $\Aut_\F(R_i)$ because $\pi_R$ contains $\Aut_\L(R_i)$ for all $i=0,\dots,n$. 
Therefore $\F_S(\pi_R)$ contains the fusion system $\F'$ on $S$ generated by $\Aut_\F(R_i)$.
Clearly $\F' \subseteq \F$ and we next claim that, in fact $\F'=\F$.
To do this, we will prove by induction on $|S:P|$ that if $\vp \in \F(P,S)$ then $\vp \in \F'(P,S)$.
If $|S:P|=1$ then $P=S=R_0$ and therefore $\vp \in \Aut_\F(R_0) \subseteq \F'$.

Assume that $|S:P|>1$.
If $P$ is not $\F$-centric or not $\F$-radical then by Alperin's fusion theorem \cite[Theorem A.10]{BLO2} it is a composition of restriction of automorphisms of $\F$-centric $\F$-radical subgroups whose order must be strictly bigger than the order of $P$ and therefore, by the induction hypothesis, all these automorphisms must belong to $\F'$.
Hence $\vp \in \F$.

If $P$ is $\F$-centric $\F$-radical then it is $\F$-conjugate to $R_i$ for some $i$.
We now claim that $P$ is $\F'$-conjugate to $R_i$.
Indeed, fix some $\F$-isomorphism $\al \colon R_i \to P$.
By Alperin's fusion theorem $\al$ is the composition of restrictions of $\F$-automorphisms $\be_1,\dots,\be_k$ of some $\F$-centric subgroups $Q_1,\dots,Q_k$, thus  $P=\be_k( \dots (\be_1(R_i)) \dots)$.
If $|Q_j|=|P|$ for some $j$ then $\be_j$ does not move $\be_{j-1}(\dots \be_1(R_i)\dots)$ and by omitting it we obtain another $\F$-isomorphism $R_i \to P$.
Thus, by possibly replacing $\al$ with another $\F$-isomorphism, we may assume that $|Q_1|, \dots,|Q_k| > |P|$, whence by induction hypothesis $\be_1,\dots,\be_k \in \F'$ and therefore $\al \in \F'(R_i,P)$.

Finally, consider $\vp \in \F(P,S)$ where $P$ is $\F$-centric $\F$-radical.
Note that $P$ and $\vp(P)$ are $\F$-conjugate to $R_i$ for some $i$.
Choose $\F'$-isomorphisms $\al' \in \F'(R_i,P)$ and $\be' \in \F'(\vp(P),R_i)$.
Then $\be' \circ \vp \circ \al' \in \Aut_\F(R_i) \subseteq \F'$ and therefore $\vp \in \F'$.
This completes the induction step.
\end{proof}

We summarize this in the following proposition.

\begin{prop}\label{prop:thmA-conditions-R}
Fix a $p$-local finite group $(S,\F,\L)$.
Then the group $\pi_R$ constructed above as the amalgamated product of \eqref{eq:robinson-type-amalgam} contains $S$ as a Sylow $p$ subgroup, it is equipped with a map $f \colon B\pi_R \to \pcomp{|\L|}$ whose restriction to $BS$ is homotopic to $\theta$ and furthermore, $\F \subseteq \F_S(\pi_R)$.
\end{prop}

We can now apply Theorem \ref{thmA}\ref{thmA:induce-fusion} to $\pi_R$ and deduce that $\F=\F_S(\pi_R)$.
Corollary \ref{corC} holds for $\pi_R$ because the implication \ref{corC:2}$\Rightarrow$\ref{corC:1} is a triviality and \ref{corC:1}$\Rightarrow$\ref{corC:2} follows from point \ref{thmA:signaliser} of Theorem \ref{thmA}.

Proposition \ref{propB} for $\pi_R$ follows from Proposition \ref{prop:sylows}\ref{prop:sylows:finite} because the graph of groups \eqref{eq:robinson-type-amalgam} has no loops.

\begin{remark}\label{rem:R-avoid-linking-system}
The groups which appear in the amalgam \eqref{eq:robinson-type-amalgam} can be constructed directly from $\F$ even without the existence of an associated linking system.
Thus, the group $\pi_R$ can be constructed for any saturated fusion system.

Indeed, since the groups $R_i$ are $\F$-centric the normalizer fusion systems $\F_i:=N_\F(R_i)$, see \cite[\S 6]{BLO2}, are constrained in the sense of \cite[Proposition  C]{BCGLO}.
Therefore $\F_i$ admit unique linking systems $\L_i$ and  there are canonically defined finite groups $L_i$ which contain $N_S(R_i)$ as a Sylow $p$-subgroup such that $\F_i=\F_{N_S(R_i)}(L_i)$.
In fact $L_i= \Aut_{\L_i}(R_i)$.
Similarly $\F_{i,0}:=N_\F(R_i<S)=N_{\F_0}(R_i)$ are constrained fusion sub-systems of $\F_0=N_\F(S)$ and therefore they have unique linking systems $\L_{i,0}$ and groups $L_{i,0}$ such that $\F_{i,0} = \F_{N_S(R_i)}(L_{i,0})$ and $L_{i,0}=\Aut_{\L_{i,0}}(R_i)$.

By construction, every morphism $P \to Q$ in $\F_0$ extends to an automorphism of $S$ and the arguments in \S\ref{subse:sfs-normaliser-decomposition} using the fact that $\hat{e} \in \L(R_i,S)$ is both a monomorphism and and epimorphism, show that there is a canonical isomorphism $\Aut_{\L_0}(R_i)=N_{L_0}(R_i)$.
Since $\F_{i,0}=N_{\F_0}(R_i)$ the uniqueness of the linking system for $\F_{i,0}$ and \cite[Definition 6.1]{BLO2} yield
\[
L_{i,0}=  \Aut_{\L_{i,0}}(R_i) = \Aut_{N_{\L_0}}(R_i) = \Aut_{\L_0}(R_i) = N_{L_0}(R_i) \leq L_0.
\]
Observe that $\F_{i,0}$ is a subsystem of $\F_i$ with the same Sylow $N_S(R_i)$ which consists of the morphisms $P \to Q$ which extend to an $\F$-automorphism of $S$.
Let $\K$ denote the collection of all the subgroups of $N_S(R_i)$ which contain $R_i$.
Since both $\F_i$ and $\F_{i,0}$ contain $R_i$ as a normal centric subgroup, $\K$ contains all the $\F_i$-centric-radical subgroups and all the $\F_{i,0}$-centric-radical subgroups.
Define a sub-category $\L_i' \subseteq \L_i$ with object set $\K$ and such that $\L_i'(P,Q)$ is the preimage of $\F_{i,0}(P,Q)$ under the projection $\L_i(P,Q) \to \F_i(P,Q)$.
It is clear that $\L_i'$ is a ``partial'' linking system for $\F_{i,0}$ in the sense that it is only defined on the subgroups in $\K$, but satisfies the axioms for a linking system.
Our argument now follows \cite[Step 7 of Theorem 4.6]{BCGLO2}.

There is a functor $\Z \colon \O(\F_{i,0})^\op \to \Ab$ defined by $P \mapsto Z(P)$.
By \cite[Proposition 3.1]{BLO2}, elements of $\varprojlim{}^3 \Z|_\K$ classify, up to isomorphism, the partial linking systems associated to $\F_{i,0}$ defined on $\K$.
Since $\K$ is closed to over-groups and contains all the $\F_{i,0}$-centric-radicals, then \cite[Proposition 3.2]{BLO2} and well known properties of the $\Lambda$-functors imply that
\[
\varprojlim_{\O(\F_{i,0}^c)}{}^3 \Z \xto{\cong} \varprojlim_{\O(\F_{i,0}^\K)}{}^3 \Z|_{\K}
\]
is an isomorphism.
By \cite[Proposition 4.2]{BCGLO} the left hand side vanishes.
This shows that the partial linking system $\L_i'$ must coincide with $\L_{i,0}^\K$ and in particular
\[
L_{i,0}=\Aut_{\L_{i,0}}(R_i) =\Aut_{\L_i'}(R_i) \leq \Aut_{\L_i}(R_i)=L_i.
\]
We obtain the following tree of groups $(G,Y)$ which is defined purely by $\F$ without any reference to an associated linking system
\begin{equation}\label{eq:R-amalgam-L-not-needed}
\xymatrix{
  & 
L_0 \ar@{-}[dl]_{L_{1,0}}  \ar@{-}[d]  \ar@{-}[dr]^{L_{n,0}}
\\
L_1 \ar@{.}[r] & L_i \ar@{.}[r] & L_n.
}
\end{equation}
and define $\pi_{R}$ as its amalgamated product.
By Proposition \ref{prop:sylows} it contains $S \leq L_0$ as a Sylow $p$-subgroup because $N_S(R_i)$ are Sylow $p$-subgroups of $L_i$ and $L_{i,0}$.
In the presence of a linking system for $\L$, this tree of groups coincides with \eqref{eq:robinson-type-amalgam}.

We can now prove that $\F=\F_S(\pi_R)$ using our geometric methods.
The assignment $[R_i] \mapsto \F_i$ and $[R_i<S] \mapsto \F_{i,0}$ define a functor $\nu$ from $\C_Y$ to the category of fusion sub-systems of $\F$.
By applying Ragnarsson's functor $\BB$ we obtain a functor $\BB\nu \colon \C_Y \to \spectra$ together with a map $\BB \nu \to \BB\F$.
Since $\F_i$ and $\F_{i,0}$ are constrained, $\BB \nu \simeq \Sigma^\infty \pcomp{B\G}$ for the functor $\G$ associated to the tree  \eqref{eq:R-amalgam-L-not-needed}.
By adjunction we obtain $B\G \to \Omega^\infty\BB \F$ in $\hotop$ and by Propositions \ref{prop:1-diml-htpy-nat-maps} and \ref{prop:sylows} we obtain $f \colon B\pi_R \to \Omega^\infty\BB \F$ whose restriction to $BS$ is homotopic to $\si_\F^\#$.
Proposition \ref{prop:stable-fusion} implies that $\F_S(\pi_R) \subseteq \F$ and the argument we used above exploiting Alperin's fusion theorem shows that $\F \subseteq \F_S(\pi_R)$.
\end{remark}

\begin{remark}\label{rem:robinson-original-vs-pir}
Robinson's originally constructed in \cite{Robinson07} a group $\Gamma$ as the amalgamated product \eqref{eq:R-amalgam-L-not-needed} but with the groups $L_{i,0}$ replaced with their Sylow $p$-subgroup $N_S(R_i)$.
The group $\pi_R$ that we construct is therefore a quotient of $\Gamma$.
Note however that if $\pi_R$ admits a signaliser functor, then $\F$ has an associated linking system and we can apply Theorem \ref{thmA} to $B\Gamma \to B\pi_R \to \pcomp{|\L|}$ in order to obtain a signaliser functor on $\Gamma$.
\end{remark}

\begin{remark}\label{rem:lib-vir-construction}
In \cite[Theorem 5.5]{Lib-Vir} it is shown that a $p$-local finite group with a ``compressible $\F$-centric'' collection $\R$ admits a group $\Gamma$ which contains $S$ as a Sylow $p$-subgroup with the property that $\F=\F_S(\Gamma)$ and that $\L$ is induced by some signaliser functor on $\Gamma$.
The group $\Gamma$ is simply the group $\pi_R$. 
Our techniques are inspired by this paper.
\end{remark}

%
%
\subsection{Leary-Stancu's construction - the subgroup decomposition}\label{subsec:Leary-Stancu}
The following argument is given explicitly in \cite{Seeliger}.
Fix a $p$-local finite group $(S,\F,\L)$ and let $\Phi=\{\vp_1,\dots,\vp_n\}$ be a set of morphisms $\vp_i \colon P_i \to S$ in $\F^c$ which, together with $\Inn(S)$,  generate $\F$ (recall Alperin's fusion theorem).
For example, $\Phi$ is the set of all the morphisms in $\F^c$.

Consider the oriented graph $Y$ with one vertex $v_0$ and oriented edges $y_1,\dots,y_n$.
Define a grapg of group $(G_,Y)$ where $G_{v_0}=S$, $G_{y_i}=G_{\bar{y_i}}=P_i$ and $y_i \mapsto P_i \xto{\vp_i} S$ and $\bar{y_i} \mapsto P_i \xto{\incl} S$.
Thus, the functor $\G \colon \C_Y \to \Gps$ is depicted by
\[
\xymatrix{
{P_1} \ar@/^0.3pc/[dr]^{\vp_1} \ar@/_0.3pc/[dr]_{\incl}  & \cdots & {P_n} \ar@/^.3pc/[dl]^{\incl} \ar@/_0.3pc/[dl]_{\vp_n}
\\
 & S  &
}
\]
Define a functor $J \colon \C_Y \to \O(\F^c)$ by
\[
J(v_0)=S, \qquad
J([y_i])=P_i, \qquad
J(\bar{y_i})= P_i \xto{\incl} S, \qquad
J(y_i)= P_i \xto{\vp_i} S.
\]
Define $\pi_{LS}:=\pi_1(G,Y,T)$ where the maximal tree $T$ must be $\{ v_0\}$.

Observe that $J^*\be \cong B\G$ in $\hotop$ where $\be$ is the subgroup decomposition, see \S\ref{subsec:sfs-subgroup-decomposition}.
Now we obtain from Proposition \ref{prop:sylows}, whose conditions are satisfied trivially, and Proposition \ref{prop:1-diml-htpy-nat-maps} and the properties of $\be$ described in \S\ref{subse:sfs-normaliser-decomposition}, a sequence of maps
\[
B\pi_{LS} \simeq \hocolim_{\C_Y}B\G \xto{\simeq} \hocolim_{\C_Y} J^*\be 
\xto{J_*} \hocolim_{\O(\F^c)} \be \xto{\simeq} |\L| \to \pcomp{|\L|}
\]
whose restriction to $BS$ is homotopic to $\theta$.
We also observe that $\F \subseteq \F_S(\pi_{LS})$ because $S \leq \pi_{LS}$ and by construction of $\pi_1(G,Y,T)$, see \S\ref{subsec:universal-cover}, every $\vp_i \colon P \to S$ in $\Phi$ is obtained by conjugation by the element $g_{y_i}$.
To summarize, we have shown:

\begin{prop}\label{prop:ls-thmA-conditions}
Fix a $p$-local finite group $(S,\F,\L)$.
Then the group $\pi_{LS}$ contains $S$ as a Sylow $p$-subgroup, it is equipped with a map $f \colon B\pi_{LS} \to \pcomp{|\L|}$ whose restriction to $BS$ is homotopic to $\theta$ and $\F \subseteq \F_S(\pi_{LS})$.
\end{prop}

Thus, we can apply Theorem \ref{thmA}\ref{thmA:induce-fusion} to recover the results of \cite{Leary-Stancu07}, namely $\F=\F_S(\pi_{LS})$.
Also Corollary \ref{corC} for $\pi_{LS}$ follows since the implication \ref{corC:2}$\Rightarrow$\ref{corC:1} is a triviality and \ref{corC:1}$\Rightarrow$\ref{corC:2} follows from point \ref{thmA:signaliser} of Theorem \ref{thmA}.

\begin{remark}
The construction of $\pi_{LS}$ appears in \cite{Leary-Stancu07}.
It is clear that it is defined purely in terms of $\F$.
The same ideas of Remark \ref{rem:R-avoid-linking-system} can be used to obtain $\F=\F_S(\pi_{LS})$ for any saturated fusion system without the assumption that a linking system exists.
\end{remark}

\subsection{Aschbacher-Chermak's construction - the normalizer decomposition}\label{subsec:Aschbacher-Chermak}

Fix an odd prime $q$ and some $n \geq 1$.
We now recall from \cite{LO1}, see also \cite{LO2}, the construction of Solomon's fusion system $\F_{\Sol}(q^n)$ at the prime $2$.
This fusion system is obtained from the fusion system of $\Spin_7(q^n)$ as follows.
Throughout we write $q^n$ for the finite field $\FF_{q^n}$ and $q^\infty$ for its algebraic closure.
There is a group homomorphism $\omega \colon \SL_2(q^\infty)^3 \to \Spin_7(q^\infty)$ whose kernel is the central subgroup of order $2$ generated by $(-I,-I,-I)$.
We will denote $\omega(X_1,X_2,X_3)=\trp[ X_1,X_2,X_3]$.

Let $H(q^\infty)$ be the image of $\omega$.
The centre of $\Spin_7(q^\infty)$ is a group of order $2$ generated by $z=\trp[ I,I,-I]$.
The set $\{ \trp[\pm I,\pm I, \pm I]\}$ forms an elementary abelian group $U$ of rank $2$.
By \cite[Lemma 2.3]{LO1} there is an element $\tau \in N_{\Spin_7(q)}(U)$ such that $c_\tau(\trp[ X_1,X_2,X_3])=\trp[X_2,X_1,X_3]$ where $c_\tau$ denotes conjugation $g \mapsto \tau g \tau^{-1}$.
Note that $\tau$ normalizes $H(q^{\infty})$.

Fix a copy of the quaternion group $Q_8$ in $\SL_2(q)$ and choose generators $A,B$ of order $4$.
Set (cf. \cite[\S 1]{LO2})
\begin{eqnarray*}
C(q^\infty) &=& \{ X \in C_{\SL_2(q^\infty)}(A) \colon X^{2^k}=I \text{ for some $k$}\} \cong \ZZ/2^\infty \\
Q(q^\infty) &=& \langle C(q^\infty),B \rangle \\
S_0(q^\infty) &=& \omega(Q(q^\infty)^3) \leq H(q^\infty) \\
S(q^\infty) &=& S_0(q^\infty) \cdot \langle \tau \rangle.
\end{eqnarray*}
We will write $\Theta(q^n)=\Theta(q^\infty) \cap \Spin_7(q^n)$ where $\Theta$ denotes one of $C(-),Q(-),S_0(-)$ or $S(-)$ etc.
By \cite[Lemma 2.7]{LO1}, $S(q^n)$ is a Sylow $2$-subgroup of $\Spin_7(q^n)$ hence also of $H(q^n) \cdot \langle \tau \rangle$.
Also $S_0(q^n)$ is a Sylow $2$-subgroup of $H(q^n)$.

There are automorphisms $\hat{\ga}$ and $\hat{\de}_u$ of $S_0(q^\infty)$ defined in \cite[Definition 1.6]{LO2} by
\[
\hat{\ga}(\trp[ X_1,X_2,X_3])=\trp[ X_3, X_1, X_2], \qquad
\hat{\de}_u(\trp[ X_1,X_2,A'B^j])=\trp[ X_1, X_2, (A')^uB^j ]
\]
where $A' \in C(q^\infty)$ and $X_1,X_2,X_3 \in Q(q^\infty)$ and $u$ is a carefully chosen $2$-adic unit.
Set
\[
\Gamma_n = \langle \Inn(S_0(q^n)) , c_\tau, \hat{\de}_u \circ \hat{\ga} \circ \hat{\de}_u^{-1} \rangle \leq \Aut(S_0(q^n)).
\]
The fusion system $\F_{\Sol}(q^n)$ on $S=S(q^n)$ is the fusion system generated by $\F_{S}(\Spin_7(q^n))$ and by $\F_{S_0(q^n)}(\Ga_n)$.

Set $S=S(q^n)$ and $S_0=S_0(q^n)$.
It easily follows from \cite[Proposition 2.5]{LO1} and the definitions that $C_S(U)=S_0$.
Also, $|S \colon S_0|=2$ because $|S(q^\infty) \colon S_0(q^\infty)|=2$ and therefore $S_0$ is fully normalized.
It is also $\F$-centric by the next proposition because by construction of $\F$, the $\F$-conjugates of $S_0$ are $\Spin_7(q^n)$-conjugates.

\begin{prop}\label{prop:spin-cent-s0}
$C_{\Spin_7(q^n)}(S_0(q^n)) \subseteq S_0(q^n)$.
\end{prop}

\begin{proof}
Since  $U \subseteq S_0$ we deduce from \cite[Proposition 2.5]{LO1} that $C_{\Spin_7(q^n)}(S_0) \subseteq C_{\Spin_7(q^n)}(U)=H(q^n)$ which contains $S_0$ as a Sylow $2$-subgroup and therefore $C_{\Spin_7(q^n)}(S_0)=Z(S_0) \times C'$ where $C'$ has odd order.
We need to show that $C'=1$.
Choose $x \in C'$ of order $r$.
Since $\ker \omega$ is central of order $2$ then $\omega^{-1}(x)$ consists of an element $\tilde{x}_1$ of order $r$ and an element $\tilde{x_2}$ of order $2r$.
Now $(Q_8)^3 \subseteq \omega^{-1}(S_0)$ must leave $\{\tilde{x}_1,\tilde{x}_2\}$ invariant because $S_0$ centralizes $x$ and since the orders of $\tilde{x}_1,\tilde{x}_2$ are different $(Q_8)^3$ centralizes these elements.
By Lemma \ref{lem:centr-q8-in-sl2q} below $\tilde{x}_1 \in Q_8$ and therefore $\tilde{x}_1=1$ which implies that $x=1$.
\end{proof}

\begin{lem}\label{lem:centr-q8-in-sl2q}
Fix an odd prime $q$.
Then
\begin{enumerate}
\item
The centralizer of any element $X$ of order $4$ in $\SL_2(q^n)$ is a cyclic group.

\item
Any quaternion group $Q_8$ in $\SL_2(q^\infty)$ contains its centralizer.
\end{enumerate}
\end{lem}

\begin{proof}
(1)
The minimal polynomial of $X$ is $x^2+1$ which splits in $\FF_{q^{n+1}}$ which contains $\zeta=\sqrt{-1}$.
Thus, $X$ is conjugate in $\GL_2(q^{n+1})$ to the diagonal matrix $\diag(\zeta,-\zeta)$ whose centralizer is the subgroup of diagonal matrices  and therefore its $\SL_2(q^{n+1})$ centralizer is the cyclic group $\FF_{q^{n+1}}^\times$.
Conjugating back we see that $C_{\SL_2(q^{n+1})}(X)$ is cyclic hence so is  $C_{\SL_2(q^{n})}(X)$.

\noindent
(2)
Fix any $n$ such that $Q_8 \leq \SL_2(q^n)$ and consider $X$ in its centralizer.
Now $Q_8$ is generated by elements $A, B$ of order $4$, so $X$ belongs to the cyclic groups $C_{\SL_2(q^n)}(A)$ and $C_{\SL_2(q^n)}(B)$.
If $\langle A \rangle \subseteq \langle X \rangle$ and  $\langle B \rangle \subseteq \langle X \rangle$ then $Q_8 \leq \langle X \rangle$ which is absurd.
So either $X \in \langle A \rangle$ or $X \in \langle B \rangle$.
\end{proof}

Set $\F=\F_{\Sol}(q^n)$ and $\L=\L_{\Sol}(q^n)$ and $Z$ is the subgroup generated by the element $z$ (this is the centre of $\Spin_7(q^n)$).
Set $\Aut_\L(S_0;Z)=C_{\Aut_\L(S_0)}(Z)$.
By \cite[Propositions 2.11, 3.3]{LO1} the linking system $\L_{\Sol}(q^n)$ contains $\L^c_S(\Spin_7(q^n))$ as a subcategory.
In fact, $\L^c_S(\Spin_7(q^n)) = C_\L(Z)$ is the centralizer linking system \cite[Definition 2.4]{BLO2} and by Proposition \ref{prop:spin-cent-s0} and  \cite[remarks before Definiton 1.6]{BLO2}
\begin{multline*}
\Aut_\L(S_0;Z)=\Aut_{C_\L(Z)}(S_0) = N_{\Spin_7(q^n)}(S_0)/C'_{\Spin_7(q^n)}(S_0)
\\
 = N_{\Spin_7(q^n)}(S_0) \leq \Spin_7(q^n).
\end{multline*}
We obtain inclusions of subcategories of $\L$
\[
\xymatrix{
\L_S(\Spin_7(q^n)) &
\B\Aut_\L(S_0;Z) \ar@{^(->}[r] \ar@{_(->}[l] &
\B\Aut_\L(S_0)
}
\]
It gives rise to the diagram of spaces below, all of whose squares commute, except the one on the top left which only commutes up to homotopy.
\[
\xymatrix{
B\Spin_7(q^n) \ar[d]_{\text{$\mod 2$ equivalence}} &
BN_{\Spin_7(q^n)}(S_0) \ar[r] \ar[l] \ar [d] &
B\Aut_\L(S_0) \ar[d]
\\
|\L_S(\Spin_7(q^n))|^{\wedge}_2 \ar[d] &
B\Aut_\L(S_0;Z)^{\wedge}_2 \ar[r] \ar[l] \ar [d] &
B\Aut_{\L}(S_0)^{\wedge}_2 \ar[d] 
\\
|\L|^{\wedge}_2 &
|\L|^{\wedge}_2 \ar@{=}[r] \ar@{=}[l] &
|\L|^{\wedge}_2.
}
\]
We observe that $N_{\Spin_7(q^n)}(S_0)$ contains the Sylow subgroup $S$ because $|S \colon S_0|=2$.
By Proposition \ref{prop:sylows}, the homotopy colimit of the first row is homotopy equivalent to $B\pi_{AC}$ where $\pi_{AC}$ is the amalgam of the tree 
\begin{equation}\label{eq:AC-amalgam}
\xymatrix@1{
{\Spin_7(q^n)} \ar@{-}[rrr]^{N_{\Spin_7(q^n)}(S_0)} & & & \Aut_\L(S_0)
} 
\end{equation}
which contains $S$ as a Sylow $2$-subgroup.
By Proposition \ref{prop:1-diml-htpy-nat-maps} there is a map $f \colon B\pi_{AC} \to |\L|{}^\wedge_2$ whose restriction to $BS$ is homotopic to $\theta$.
Moreover, $\Spin_7(q^n)$ and  $\Aut_\L(S_0)$ are subgroups of $\pi_{AC}$ and also $\Aut_\F(S_0)=\Aut_{\Aut_\L(S_0)}(S_0)$ contains $c_\tau$ and $\hat{\de}_u \circ \hat{\ga} \circ \hat{\de}_u^{-1}$ because these homomorphisms clearly normalize $S_0(q^n)$.
We therefore deduce that $\F \subseteq \F_S(\pi)$.
Thus we have proven:

\begin{prop}\label{prop:ac-thmA}
The group $\pi_{AC}$ satisfies the hypotheses of Theorem \ref{thmA}.
In particular $\F=\F_S(\pi_{AC})$ and $\pi_{AC}$ has a signaliser functor which induces $\L$.
It also satisfies the condition of Proposition \ref{propB} because \eqref{eq:AC-amalgam} has no loops.
\end{prop}

\begin{remark}\label{rem:improved-ac}
The group $\Gamma$ that Aschbacher and Chermak constructed in \cite{Asch-Chermak} is the amalgamated product 
\[
\xymatrix@1{
{\Spin_7(q^n)} \ar@{-}[rrr]^{N_{\Spin_7(q^n)}(U)} & & & K
}
\]
where $|K \colon N_{\Spin_7(q^n)}(U)|=3$.   
They need the hypothesis that $q \cong 3 \text{ or } 5 \mod 8$, however, which we don't.

Indeed, the groups $\Gamma$ and $\pi_{AC}$ are different.
We will show below that $U$ is the unique normal subgroup of $S_0=S_0(q^n)$ which is isomorphic to the Klein group $\ZZ/2\oplus \ZZ/2$ and therefore $N_{\Spin_7(q^n)}(S_0) \subseteq N_{\Spin_7(q^n)}(U)$.
The inclusion is proper because $H(q^n)=C_{\Spin_7(q^n)}(U)$ does not normalize $S_0(q^n)$.
In addition, $\Aut_{\Spin_7(q^n)}(U)=C_2$ and $\Aut_{\F_{\Sol}(q^n)}(U)=\Si_3$ and $\Aut_{\F_{\Sol}(q^n)}(S_0)$ contains $\hat{\de}_u\hat{\ga}\hat{\de}_u^{-1}$ from which we deduce that $|\Aut_\L(S_0) \colon N_{\Spin_7(q^n)}(S_0)|=3$.
Thus the groups we use in the amalgam \eqref{eq:AC-amalgam} are smaller than those used by Aschbacher and Chermak.
It is interesting to see how our approach interprets the signaliser functors geometrically in Theorem \ref{thmA} rather than algebraically.

\begin{claimx}
$U$ is the unique normal subgroup of $S_0(q^n)$ isomorphic to the Klein group.
\end{claimx}

\begin{proof}
Consider a Klein group $E \lhd S_0$ and set $\tilde{E}=\omega^{-1}(E)$.
Note that $\tilde{E}$ is normalized by $(Q_8)^3 \leq Q(q^\infty)^3$.
Fix an element $\tilde{X}=(X_1,X_2,X_3) \in \tilde{E}$ and note that $X_i^2=\pm I$ because all the elements of $E$  have order $2$ and $\ker \omega=\pm(I,I,I)$.
If $X_i$ has order $4$ then $X_1^2=X_2^2=X_3^2=-I$ so $X_1,X_2,X_3$ have order $4$ and Lemma  \ref{lem:centr-q8-in-sl2q} implies that $\tilde{E}$ contains at least $2^3$ different $(Q_8)^3$ conjugates of $\tilde{X}$.
This is impossible, so the order of $\tilde{X}$ is at at most $2$.

If $E \neq U$ then we may assume that $X_1 \neq \pm I$ hence $X_1=WB^{\pm 1}$ for some $W \in C(q^\infty)$.
Now it follows that $\tilde{X} \cdot B\tilde{X} B^{-1} = (-W^2,I,I) \in \tilde{E}$ must also have order $2$.
This can only happen if $W \in \langle A \rangle \leq C(q^\infty)$ which implies that $X_1 \in Q_8$ whose only element of order $2$ is $-I$ and this is absurd.
\end{proof}
\end{remark}

\begin{remark}
The construction which we have presented here is motivated by the normalizer decomposition with respect to the collection of the elementary abelian subgroups $\E$, see \cite{Lib06}.
If $\D$ is the subposet of $\bar{S}(\E)$ of the form $[Z] \leftarrow [Z < U] \to [U]$ then the restriction of the normalizer decomposition to $\D$ has the form
\[
|C_\L(Z)| \longleftarrow |C_\L(U)|_{h \Aut_\F(Z<U)} \longrightarrow |\C_\L(U)|_{h\Aut_\F(U)}.
\]
Up to $2$-completion, the space on the left is $B\Spin_7(q^n)$.
A careful examination of the constructions in \cite[Theorem B]{Lib06} shows that these maps are induces by inclusion of categories which contain $\B\Aut_{C_\L(Z)}(S_0)$ and $\B\Aut_\L(S_0)$ as very natural full subcategories to look at.
We omit the details.
\end{remark}


\end{document}